\documentclass[11pt,reqno,palentino]{amsart} 

\usepackage{mypack}

\usepackage{latexsym}
\usepackage{a4wide}
\usepackage{amscd}
\usepackage{graphics}
\usepackage{amsmath}
\usepackage{amssymb}
\usepackage{mathrsfs}
\usepackage{accents}
\usepackage{enumerate}

\usepackage[backgroundcolor=white,linecolor=gray]{todonotes}%Adds\todo, which is nice.

\graphicspath{{figure/}} %%Figures should be stored here 

\DeclareMathOperator{\codim}{codim}
\DeclareMathOperator{\hess}{Hess}
\DeclareMathOperator{\sign}{sign}

\newtheorem{theorem*}{Theorem}

\author{ Thomas O. Rot, Maciej Starostka, Nils Waterstraat}
\begin{document}
%\todo{subject classifier etc}

\title{The relative cup-length in local Morse cohomology}
\begin{abstract}
Local Morse cohomology associates cohomology groups to isolating neighborhoods of gradient flows of Morse functions on (generally non-compact) Riemannian manifolds $M$. We show that local Morse cohomology is a module over the cohomology of the isolating neighborhood, which allows us to define a cup-length  relative to the cohomology of the isolating neighborhood that gives a lower bound on the number of critical points of functions on $M$ that are not necessarily Morse. Finally, we illustrate by an example that this lower bound can indeed be stronger than the lower bound given by the absolute cup-length. 
\end{abstract}

\maketitle 
\section{Introduction}

% The cup-length of a closed manifold $M$ gives an estimate for a lower bound of a number of critical points that any function $f$ on $M$ has to have.
Let $f$ be a Morse function on a closed manifold $M$. For a generic choice of a metric $g$ on $M$, one can define Morse cohomology, which is isomorphic to the singular cohomology of $M$, and consequently its isomorphism class does neither depend on $f$ nor $g$. The generators of the Morse complex are the critical points of $f$, and the ranks of the cohomology groups of $M$ give lower bounds on the minimal number of critical points of Morse functions on $M$. These lower bounds are captured by the celebrated Morse inequalities.

There are also bounds for the number of critical points of possibly degenerate functions. Via the Lyusternik-Schnirelman category, the cup-length gives a lower bound on the number of critical points for any function (not necessarily Morse) on $M$. The cup-length can be obtained by a well known Morse-theoretic interpretation of the cup product, and as the Morse cup length agrees with the cup length in the singular cohomology of $M$, it gives a lower bound for the number of critical points of any function on $M$.

On non-compact manifolds Morse cohomology in general is not even well-defined, as the boundary operator does not need to square to zero. Moreover, even if it is well-defined, the Morse cohomology will depend on the function and the metric used to define it. In fact any open manifold admits a function without critical points, cf.~\cite[Theorem 4.8]{Hi}, and thus the Morse cohomology of this function vanishes. Consequently, additional assumptions are needed to get interesting bounds for critical points on non-compact manifolds. Local Morse cohomology is a possible way to deal with this issue. Local Morse cohomology is defined for isolating neighborhoods of the gradient flow of a Morse function in possibly non-compact manifolds (see Section~\ref{sec:loc} for definitions). It only depends on the isolated homotopy class of the gradient flow of the Morse function, and is isomorphic to the cohomological Conley index, cf.~\cite[Theorem 7.1]{RV}.

Using a cup product for local Morse cohomology one also gets lower bounds for the number of critical points of functions whose gradient flows are isolated homotopic. However, the cup-length estimate obtained in this way may be very far from optimal. For example, let $f:\mC \mP^n\rightarrow \mR$ be a perfect Morse function on the complex projective space. Then all critical points of $f$ have even index. Define the function $F: \mC \mP^n \times \mR \to \mR$ by $F(x,y) = f(x) - y^2$ and let $N=\mC \mP^n\times [-1,1]$ be an isolating neighborhood of the gradient flow with respect to the obvious metric. All the critical points of $F$ now have odd index, because there is one extra unstable direction at the critical points. The local Morse cohomology is non-zero only in the odd dimensions $1,3,\ldots, 2n+1$. As the cup product maps two odd dimensional classes to an even dimensional class, the cup product is trivial in this case and thus the cup-length does not provide any information on the number of critical points.

However, a perturbation of $F$, whose gradient flow preserves isolation, will have at least $\CL(\mC \mP^n) + 1 = n+1$ critical points. This can be shown, e.g., by applying the relative cup-length for the Conley index as defined in \cite{KGUss}. The aim of this paper is to give a purely Morse cohomological explanation of the latter fact. We first define a Morse theoretic relative cup-length in Definition~\ref{def:relcup} and then we prove our main Theorem~\ref{thm:main}, which states that this relative cup-length gives a lower bound for the number of critical points of any (not necessarily Morse) function. In Section~\ref{sec:example} we explicitly compute the absolute and relative cup-length in Morse cohomology on $\mR\mP^n\times \mR$ to exemplify our approach.

% An advantage of using Morse theory instead of the Conley index lies in possible generalizations. One can try to extend the definition of the relative cup-length and an appropriate estimate to the case of Hilbert manifolds. So far there is no Conley index theory on Hilbert manifolds while the Morse theory on Hilbert manifolds proved to be a powerful tool in differential and symplectic geometry. However, that question goes beyond the scope of this paper.

\section{Local Morse cohomology}
\label{sec:loc}
The aim of this section is to recall the definition of local Morse cohomology, mostly to introduce our notation. Let $M$ be a smooth manifold which is not necessarily closed or orientable. Let $f:M\rightarrow \mR$ be a smooth function and $g$ a metric on $M$. Henceforth we denote the gradient flow generated by the vector field $-\nabla_gf$ by $\varphi$. We will always assume that the vector fields are complete, so that they define global flows. A subset $N\subset M$ is an \emph{isolating neighborhood} if it is compact and the \emph{invariant set}  $\mathrm{Inv}(\varphi,N)=\{x\in N\,|\, \varphi_t(x)\in N\text{ for all }t\in \mR\}$ is contained in the interior of $N$. The \emph{local stable manifold} $W^s(x;N)$ and \emph{local unstable manifold} $W^u(x;N)$ of a critical point $x\in N$ of $f$ are defined by
  \begin{align}\label{stableunstable}
  \begin{split}
    W^{s}(x;N)&=\{p\in \mathrm{Int}(N)\,|\, \varphi_t(p)\in N \quad \text{for all} \quad t\geq 0\quad
 \text{and} \lim_{t\rightarrow \infty} \varphi_t(p)=x\},\\
  W^{u}(x;N)&=\{p\in \mathrm{Int}(N)\,|\, \varphi_t(p)\in N\quad  \text{for all} \quad t\leq 0\quad
 \text{and} \lim_{t\rightarrow -\infty} \varphi_t(p)=x\}.
\end{split}
\end{align}
We say that the pair $(f,g)$ is \emph{Morse-Smale} on $N$ if all critical points of $f$ in $N$ are non-degenerate and their local stable and unstable manifolds \eqref{stableunstable} intersect transversely. Note that this is a generic condition. The (Morse) index of a non-degenerate critical point $x$ will be denoted by $\ind x$.

The local stable and unstable manifolds are contractible manifolds, and thus they are both orientable and coorientable, i.e.~their normal bundles are orientable. Let $\ko$ be a choice of orientation of the local unstable manifolds. We call the quadruple $\cQ=(N,f,g,\ko)$ a \emph{local Morse datum}. We denote the local stable and unstable manifolds of the local Morse datum by $W^{s}(x;\cQ)$, $W^{u}(x;\cQ)$ respectively. The orientations $\ko$ of the local unstable manifolds induce coorientations of the local stable manifolds as the normal space of $W^s(x;\cQ)$ at the critical point $x$ can be identified with the tangent space of $W^u(x;\cQ)$ at $x$, which is oriented by the choice of $\ko$. This implies that the moduli space of parametrized orbits $W(x,y;\cQ)=W^u(x;\cQ)\cap W^s(y;\cQ)$ is oriented, and similarly the space of unparameterized orbits $M(x,y;\cQ)=W(x,y;\cQ)/\mR$ is oriented, where $\mR$ acts on $M(x,y;\cQ)$ by the gradient flow. In the appendix we discuss our orientation conventions. If $\ind x=\ind y+1$, the moduli space $M(x,y;\cQ)$ is zero dimensional. We denote by $m_x^y=m_x^y(\cQ)$ the oriented count of the points in $M(x,y;\cQ)$. Associated to a local Morse datum is the \emph{local Morse complex} $CM_*(\cQ)$. The local Morse complex is freely generated by the critical points of $f$ in $N$, graded by the index, and the differential is defined by $\partial x=\sum m_x^yy$ where the sum runs over all critical points $y$ in $N$ with $\ind y=\ind x-1$. Standard arguments show that this is indeed a chain complex, cf.~\cite{R,RV}, and the \emph{local Morse homology} $HM_*(\cQ)$ is the homology of this complex.

Let $N$ be an isolating neighborhood of two flows $\varphi^\alpha$ and $\varphi^\beta$. If there exists a homotopy $\varphi^\lambda$, $\lambda\in [0,1]$ of flows between $\varphi^\alpha$ and $\varphi^\beta$, such that $N$ is an isolating neighborhood of all $\varphi^\lambda$ simultaneously, then the flows $\varphi^\alpha$ and $\varphi^\beta$ are said to be \emph{isolated homotopic}. If the gradient flows of two local Morse data $\cQ^\alpha$ and $\cQ^\beta$ are isolated homotopic, then there is a canonical isomorphism $HM_*(\cQ^\alpha)\cong HM_*(\cQ^\beta)$, cf.~\cite{R,RV}. The \emph{local Morse cohomology} $HM^*(\cQ)$ is defined as the homology of the dual complex $CM^*(\cQ):=\mathrm{Hom}(CM_*(\cQ),R)$ with (co)differential $\delta=\partial^*$, where $R$ is a unital commutative ring. We will surpress $R$ in the notation. If $x\in N$ is a critical point of $f$, i.e. a generator of $CM_*(\cQ)$, then we write $\eta^x$ for the dual generator in $CM^*(\cQ)$ defined by $\eta^x(y)=1$ if $x=y$ and $\eta^x(y)=0$ if $x\not =y$. The codifferential is then $\delta \eta^x=\sum m_y^x\eta^y$ where the sum runs over all critical points $y$ of $f$ in $N$ with $\ind y=\ind x+1$. If $\cQ^\alpha$ and $\cQ^\beta$ denote Morse data on different manifolds and $h^{\beta\alpha}:M^\alpha\rightarrow M^\beta$ is a map such that $h^{\beta\alpha}\bigr|_{W^u(x;\cQ^\alpha)}$ is transverse to $W^s(y;\cQ^\beta)$ for all critical points $x$ and $y$, then there is an induced map $h^{\beta\alpha}:H^*(\cQ^\beta)\rightarrow H^*(\cQ^\alpha)$ defined by $h^{\beta\alpha}(\eta^y)=\sum n_{h^{\beta\alpha}}(x,y)\eta^x$, where the sum runs over all critical points $x$ with $\ind x=\ind y$ and $n_{h^{\beta\alpha}}(x,y)$ is the oriented count of the zero dimensional manifold $W^{u}(x;\cQ^\alpha)\cap (h^{\beta\alpha})^{-1}W^s(y;\cQ^\beta)$. This definition has the expected functoriality and homotopy invariance properties. The situation in local Morse homology is a bit more subtle as it requires $h^{\beta\alpha}$ to behave properly with respect to the isolating neighborhoods and the flows. We refer to~\cite{RV2} for functoriality statements for local Morse cohomology, which will be used below.

\subsection{Cup product in Morse cohomology}

Before we define cup products in local Morse cohomology we recall their definition in Morse cohomology. We therefore assume for a moment that $M$ is closed and that $\cQ^\alpha,\cQ^\beta,\cQ^\gamma$ are \emph{global} Morse data on $M$ (i.e. we assume that $N^\alpha=N^\beta=N^\gamma=M$). The Morse theoretic cup product is a map $ \smallsmile:HM^k(\cQ^\alpha)\times HM^l(\cQ^\beta)\rightarrow HM^{k+l}(\cQ^\gamma)$ which is defined on the chain level via the formula
$\eta^{x}\smallsmile \eta^{y}=\sum w_z^{x,y} \eta^z$, where the sum runs over all critical points $z$ with $\ind{z}=\ind{{x}}+\ind{{y}}$ and where $w^{{x},{y}}_z:=w^{x,y}_z(\cQ^\gamma,\cQ^\alpha,\cQ^\beta)$ is the oriented count of the points in the oriented zero dimensional manifold \[W(z,x,y):=W(z,x,y,\cQ^\gamma,\cQ^\alpha,\cQ^\beta):=W^u(z;\cQ^\gamma)\cap W^s({x};\cQ^\alpha)\cap W^s({y};\cQ^\beta).\] Of course for this formula to be well defined the manifolds $W^u(z;\cQ^\gamma), W^s({x};\cQ^\alpha),$ and $W^s({y};\cQ^\beta)$ must intersect mutually transversally. By a perturbation this is in general possible, and one can take in fact $\cQ^\alpha=\cQ^\gamma$. However the remaining Morse datum $\cQ^\beta$ must be chosen different from $\cQ^\alpha=\cQ^\gamma$, if there is any chance to achieve transversality. The map $\smallsmile$ satisfies the relation
\begin{equation}
  \label{eq:cup}
\delta (\eta^{x}\smallsmile \eta^{y})-(\delta \eta^{x})\smallsmile \eta^{y}-(-1)^{\ind {x}}\eta^{x}\smallsmile \delta \eta^{y}=0
\end{equation}
which can be proven by compactifying the (non-compact) one-dimensional moduli space $W(z,x,y)$ for $\ind z=\ind {x}+\ind {y}+1$ to a one-dimensional manifold with boundary. The non-compactness of this moduli space comes from breaking, just as the non-compactness of $M({x},{y})$ comes from points escaping either $W^u({x})$ or $W^s({y})$. The non-compact ends of $W(z,x,y)$ occur when a sequence of points in $W(z,x,y)$ runs out of either $W^u(z;\cQ^\gamma)$,  $W^s({x};\cQ^\alpha)$ or $W^s({y};\cQ^\beta)$, see Figure \ref{fig:breaking}. % The point is that, when $\ind y=\ind {x_1}+\ind {y}+1$ the space $W(y,x_1,y)$ can be compactified to a manifold with boundary.

The boundary components of the compactified moduli space are zero-dimensional and can be enumerated by
\[M(z,z') \times W(z',x,y),\quad M(x',x) \times W(z,x',y), \quad \text{and}\quad(-1)^{\ind {x}}M(y',y)\times W(z,x,y') ,\] for all possible critical points $x',y',z'$ such that $\ind{x'}=\ind{x}+1$, and $\ind{z'}=\ind z-1$. The oriented count of the boundary components must therefore be zero, which is exactly what Equation~\eqref{eq:cup} expresses.

\begin{figure}
  \includegraphics[width=.6\textwidth]{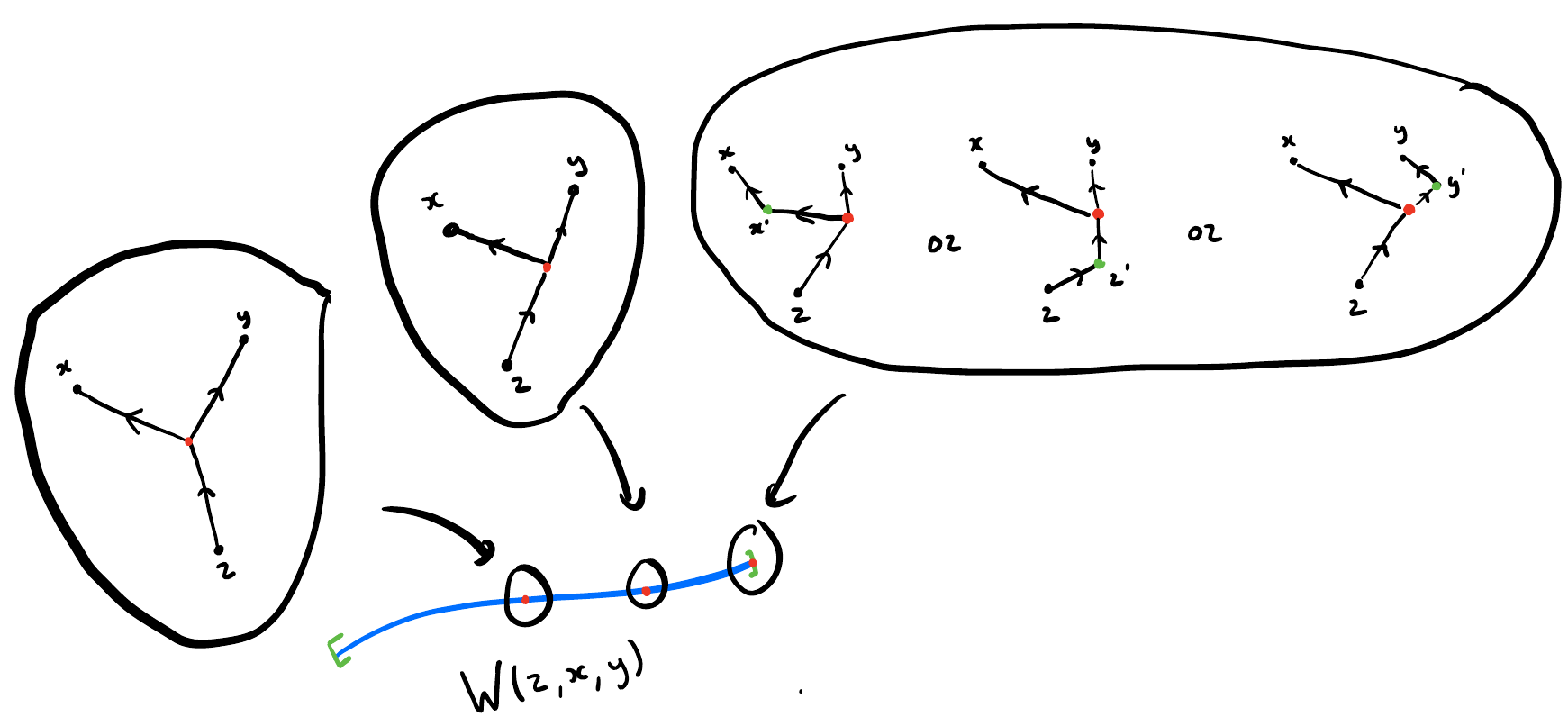}
  \caption{The one-dimensional moduli space $W(z,x,y)$ (in blue) is typically not compact. The non-compact ends are parametrized by the three types of breaking that can occur. Equation \eqref{eq:cup} expresses that the moduli space can be compactified to a one dimensional manifold with boundary by adjoining these possible breakings at the non-compact ends. }
  \label{fig:breaking}
\end{figure}
By our orientation conventions of the moduli space we immediately get that
\begin{equation}
  \label{eq:com}
\eta^{x}\smallsmile \eta^{y}=(-1)^{\ind {x}\ind {y}}\eta^{y}\smallsmile \eta^{x}
\end{equation}
Where on the left hand side we have the cup-product $\smallsmile:HM^*(\cQ^\alpha)\times HM^*(\cQ^\beta)\rightarrow HM^*(\cQ^\gamma)$ while on the right hand side we consider the cup-product $\smallsmile:HM^*(\cQ^\beta)\times HM^*(\cQ^\alpha)\rightarrow HM^*(\cQ^\gamma)$. In this sense the cup-product is graded commutative. The cup-product is also associative on the homology level, which is trickier to prove. As we do not need the associativity for our main result, we do not furnish further details about it. 
The Morse data $\cQ^\alpha$ and $\cQ^\beta$ give a Morse datum $\cQ^\alpha\times \cQ^\beta$ on the product manifold $M\times M$, where the Morse function $f$ of $\cQ^\alpha\times \cQ^\beta$ is given by $f(x,y)=f^\alpha(x)+f^\beta(y)$ and the metric is the product metric. By mapping a pair $\eta^{x}\in H^{k}(\cQ^\alpha)$ and $\eta^{y}\in H^{k}(\cQ^\beta)$ to the critical point $\eta^{(x,y)}$ of $f$ we obtain the Morse theoretic cross product $\times:H^k(\cQ^\alpha)\times H^{l}(\cQ^\beta)\rightarrow HM^{k+l}(\cQ^\alpha\times \cQ^\beta)$. We have the equality $
  \eta^{x}\smallsmile \eta^{y}=\Delta^*(\eta^x\times\eta^y),$
  where $\Delta:M\rightarrow M\times M$ denotes the diagonal, see also~\cite{AS, RV2}. 

  The cup product commutes with the canonical isomorphisms. Suppose that $\cQ^{\alpha'},\cQ^{\beta'},\cQ^{\gamma'}$ are other global Morse data on $M$. Then we have a commutative diagram

  \begin{equation}
    \label{eq:naturality}
    \begin{split}
    \xymatrix{HM^{k}(\cQ^\alpha)\times HM^{l}(\cQ^\beta)\ar[r]^-\times&HM^{k+l}(\cQ^\alpha\times \cQ^\beta)\ar[r]^-{\Delta^*}&HM^{k+l}(\cQ^\gamma)\\
      HM^{k}(\cQ^{\alpha'})\times HM^{l}(\cQ^{\beta'})\ar[u]\ar[r]^-\times&HM^{k+l}(\cQ^{\alpha'}\times \cQ^{\beta'})\ar[u]\ar[r]^-{\Delta^*}&HM^{k+l}(\cQ^{\gamma'}) \ar[u]}
  \end{split}
\end{equation}
The horizontal compositions are the cup products and the vertical arrows are the canonical isomorphisms induced by continuation. It follows that the cup product is natural with respect to maps, which can also be proven directly.

\subsection{Cup products in local Morse Cohomology}

The cup product in global Morse theory, i.e.~when the isolating neighborhood $N^{\alpha}=N^{\beta}=N^\gamma=M$ is a the whole closed manifold $M$, is well-defined under the assumption that the manifolds $W^u(z; \cQ^{\gamma}),W^s(x; \cQ^{\alpha}),$ and $ W^s(y;\cQ^{\beta})$ intersect transversely. For local Morse cohomology an additional isolation condition is needed that involves the flows $\varphi^{\alpha},\varphi^{\beta}$, and $\varphi^\gamma$ simultaneously.

\begin{definition}
Suppose $N$ is an isolating neighborhood of the flows $\varphi^{\alpha},\varphi^{\beta},\varphi^\gamma$. The flows are said to be \emph{isolation compatible} if the set
  \[
Z(\cQ^\gamma,\cQ^{\alpha},\cQ^{\beta})=\{p\in N\,|\, \varphi^{\alpha}_t(p),\varphi^{\beta}_t(p),\varphi^\gamma_{-t}(p)\in N\quad \text{for all} \quad t\geq 0\}
\]
is properly contained in the interior of $N$. Three local Morse data $\cQ^{\alpha},\cQ^{\beta}$ and $\cQ^\gamma$ are \emph{isolation compatible} if $N:=N^{\alpha}=N^{\beta}=N^\gamma$ and their gradient flows are isolation compatible. We call the local Morse data \emph{mutually transverse} if the manifolds $W^s(x;\cQ^{\alpha})$, $W^s(y;\cQ^{\beta}),$ and $W^u(z;\cQ^\gamma)$ intersect transversally. 
\end{definition}

The set \[
  Z:=Z(\cQ^\gamma,\cQ^{\alpha},\cQ^{\beta})=\bigcap_{t\geq 0}\left(\varphi^{\alpha}_{-t}(N)\cap\varphi^{\beta}_{-t}(N)\cap\varphi^\gamma_{t}(N)\right)
\]
is closed as it is the intersection of closed sets. Consequently, as $N$ is compact, $Z$ is compact. Let us note that that isolation compatibility is an open condition in the space of flows. % The set $Y(\cQ^\gamma,\cQ^{\alpha},\cQ^{\alpha_2})$ is closed and contained in the interior of $N$.
Three flows are isolation compatible if the diagonal map $\Delta: M\rightarrow M\times M$ is an isolated map in the language of~\cite{RV2}. Here the domain is equipped with the flow $\varphi^\gamma$ and isolating neighborhood $N$ and the codomain is equipped with the flow $\varphi^{\alpha}\times \varphi^{\beta}$ with isolating neighborhood $N\times N$. 

\begin{lemma}
  \label{lem}
Suppose $N$ is an isolating neighborhood of flows  $\varphi^{\alpha},\varphi^{\beta}$, and $\varphi^\gamma$. Suppose that $\varphi^{\alpha}=\varphi^\gamma$ or $\varphi^{\beta}=\varphi^\gamma$. Then the flows are isolation compatible. 
\end{lemma}
\begin{proof}
  Without loss of generality, we can assume $\varphi^{\alpha}=\varphi^\gamma$. Thus $Z(\cQ^{\alpha}, \cQ^{\alpha}, \cQ^{\beta})\subset \mathrm{Inv}(\varphi^{\alpha};N)$ which is contained in the interior of $N$ as $N$ is an isolating neighborhood of the flow $\varphi^\alpha$.
\end{proof}

% \red{Recall that an isolating neighborhood $N^\alpha$ of $\varphi^\alpha$ is an \emph{attracting neighborhood} if $\varphi_t^\alpha(N^\alpha)(N^\alpha)\subset \mathrm{int}(N^\alpha)$ for all $t>0$. 
% \begin{lemma}
% Suppose $N$ is an isolating neighborhood for flows $\varphi^\alpha,\varphi^{\alpha_2},\varphi^\beta$ and suppose that $N$ is an attracting neighborhood of either $\varphi^\alpha$ or $\varphi^{\alpha_2}$. Then the flows are isolation compatible. 
% \end{lemma}
% \begin{proof}
% In red is IS FALSE!!
% \end{proof}}
\begin{proposition}
  \label{prop:all}
  Suppose that the local Morse data $\cQ^{\alpha},\cQ^{\beta}$ and $\cQ^\gamma$ are isolation compatible and mutually transverse. Then the cup-product $\eta^x\smallsmile \eta^y=\sum_{\ind{z}=\ind{x}+\ind{y}}w_z^{x,y}\eta^z$ is well defined on the chain level and satisfies Equation \eqref{eq:cup}. Thus it descends to a map $\smallsmile:HM^k(\cQ^{\alpha})\times HM^l(\cQ^{\beta})\rightarrow HM^{k+l}(\cQ^\gamma)$. The cup product is graded commutative in the sense of Equation \eqref{eq:com}. Moreover if $\cQ^{\alpha'},\cQ^{\beta'}$ and $\cQ^{\gamma'}$ are other isolation compatible local Morse data on the same isolating neighborhood such that the flows $\varphi^{\alpha'},\varphi^{\beta'}$ and $\varphi^{\gamma'}$ are isolated homotopic to $\varphi^{\alpha},\varphi^{\beta}$ and $\varphi^{\gamma}$ respectively, then the diagram
  \begin{equation}
    \label{eq:naturality2}
    \begin{split}
    \xymatrix{HM^{k}(\cQ^{\alpha})\times HM^{l}(\cQ^{\beta})\ar[r]^-\smallsmile&HM^{k+l}(\cQ^\gamma)\\
      HM^{k}(\cQ^{\alpha'})\times HM^{l}(\cQ^{\beta'})\ar[u]\ar[r]^-\smallsmile&HM^{k+l}(\cQ^{\gamma'}) \ar[u]}
  \end{split}
\end{equation}
is commutative, where the vertical maps are induced by continuation.
\end{proposition}
\begin{proof}
  In the local case Equation \eqref{eq:cup} is derived, as in the global case, by investigation of the non-compactness of the moduli space $W(z,x,y)$. The issue in the local situation is that a priori the breaking  could occur outside of $N$. Note that $Z=\bigcup_{x',y',z'}W(z',x',y')$ and that $Z$ is compact. Hence if $p_n$ is a sequence in $W(z,x,y)$, then there is a subsequence converging to $p\in W(z',x',y')$ for some critical points $x',y',z'$. Let us consider this subsequence. By the transversality assumption, $\ind{z'}\leq \ind z$, $\ind {x'}\geq \ind x$, $\ind{y'}\geq \ind y$ with at most one of the inequalities strict. If an inequality is strict e.g.~$\ind{z'}<\ind z$, then the index difference $\ind z-\ind{z'}$ is one. If the inequality is not strict, e.g.~$\ind{z'}=\ind z$, then these are the same critical points, i.e.~$z'=z$. Now assume that $\ind{z'}=\ind z -1$ (the cases $\ind{ x'}=\ind x+1$ and $\ind{y'}=\ind y+1$ can be treated in a similar way). Then $\varphi^\gamma_t(p_n)\in W^u(z;\cQ^\gamma)$ for all $n$ and $t<0$, while $\varphi^\gamma_t(p)\in W^u(z';\cQ^\gamma)$. Then there exist times $t_n<0$ such that $\lim_{n\rightarrow \infty} \varphi^\gamma_{t_n}(p_n)\in W^u(z;\cQ^\gamma)\cap W^s(z';\cQ^\gamma)$. Thus in this case the non-compact end of $W(z;x,y)$ is enumerated by $M(z,z')\times W(z',x,y)$.

The commutativity of the diagram in the proposition follows from the local version of Diagram ~\eqref{eq:naturality}. This diagram is well defined and commutative if the map $\Delta$ is isolated throughout the isolated homotopies of flows, cf.~\cite{RV2}, which holds by the assumptions of the proposition.
\end{proof}

\section{The relative cup length in local Morse cohomology}

To define the relative cup-length, we need a special class of local Morse data.

\begin{definition}
  A local Morse datum $\cQ=(f,g,N,\ko)$ is \emph{attracting} if there exists $c\in \mR$ such that
  \[
    f(x)\begin{cases} <c\qquad \text{if} \qquad &x\in \mathrm{Int} N\\
      =c& x\in \partial N\\
      >c&x\in M\setminus N.
      \end{cases}
  \]
\end{definition}

All attracting local Morse data are isolated homotopic to each other, hence their local Morse homologies are isomorphic. The Morse cohomology of an attracting local Morse datum $\cQ=(f,g,N,\ko)$ is isomorphic to the singular cohomology of $N$, i.e.~$HM^*(\cQ)\cong H^*(N;R)$, see for example~\cite{KM,RV}. 

Let $\cQ^\alpha$ be a local Morse datum, and $\cQ^{\beta}$ be an attracting local Morse datum such that $\cQ^\alpha,\cQ^\beta,\cQ^\alpha$ are mutually transverse. By Lemma~\ref{lem} the flows are isolation compatible. Thus we have a well defined cup product $\smallsmile:HM^*(\cQ^\alpha)\times HM^*(\cQ^{\beta})\rightarrow HM^*(\cQ^{\alpha})$.

% If this cup-product is non-trivial, then for any other attracting local Morse datum $\cQ^{\alpha}$ such that $\cQ^\alpha,\cQ^{\alpha},\cQ^\alpha$ is mutually transverse the cup-product $\smallsmile:HM^*(\cQ^\alpha)\times HM^*(\cQ^{\alpha})\rightarrow HM^*(\cQ^{\alpha})$ is non-trivial.

% This gives rise to the following definition.
% In this definition the cup products are maps $\smallsmile:HM^*(\cQ^\alpha)\times HM^*(\cQ^{\beta})\rightarrow HM^*(\cQ^{\alpha})$.
% \begin{definition}
%   \label{def:relcup}
%   Let $\cQ^\alpha$ be a local Morse datum, and $\cQ^{\beta}$ an attracting local Morse data such that $\cQ^\alpha,\cQ^{\beta},\cQ^\alpha$ are mutually transverse.
%   The \emph{relative cup-length} $Y(\cQ^\alpha)$ of $\cQ^\alpha$ is defined to be 
%   \begin{itemize}
%   \item $Y(\cQ^\alpha)=0$ if $HM^*(\cQ^\alpha)=0$.
%   \item $Y(\cQ^\alpha) =1$ if $HM^*(\cQ^\alpha)\not =0$ but $\eta\smallsmile \mu_1=0$ for all $\eta\in HM^*(\cQ^\alpha)$ and all $\mu_1\in HM^{l_1}(\cQ^{\beta})$ with $l_1>0$.
%   \item $Y(\cQ^\alpha)=n \geq 2$ if there exists a class $\eta \in HM^k(\cQ^\alpha)$ and classes $\mu_i\in HM^{l_i}(\cQ^\beta)$ with $l_i>0$, for $1\leq i\leq n-1$ such that $((\ldots((\eta\smallsmile \mu_1)\smallsmile \mu_2)\ldots )\smallsmile\mu_{n-1})\not=0$ but for any classes $\eta'\in HM^{k'}(\cQ^\alpha)$ and $\mu'_i\in  HM^{l_i'}(\cQ^\beta)$ with $1\leq i\leq n$ and $l_i'>0$ we have that $(\ldots(\eta'\smallsmile \mu_1')\ldots \smallsmile\mu_n')=0$.
%   \end{itemize}
% \end{definition}

\begin{definition}
  \label{def:relcup}
  Let $\cQ^\alpha$ be a local Morse datum, and $\cQ^{\beta}$ an attracting local Morse data such that $\cQ^\alpha,\cQ^{\beta},\cQ^\alpha$ are mutually transverse.
  The \emph{relative cup-length} $Y(\cQ^\alpha)$ of $\cQ^\alpha$ is defined to be 
  \begin{itemize}
  \item $Y(\cQ^\alpha)=0$ if $HM^*(\cQ^\alpha)=0$.
  \item $Y(\cQ^\alpha) =1$ if $HM^*(\cQ^\alpha)\not =0$ but ${\eta\smallsmile \mu_1=0}$ for all $\eta\in HM^*(\cQ^\alpha)$ and ${\mu_1\in HM^{l_1}(\cQ^{\beta})}$ with $l_1>0$.
  \item $Y(\cQ^\alpha)=n \geq 2$ if $n$ is the maximum number such that there exist classes ${\eta \in HM^k(\cQ^\alpha)}$ and ${\mu_i\in HM^{l_i}(\cQ^\beta)}$ with $l_i>0$, for $1\leq i\leq n-1$ for which
    \[(\ldots((\eta\smallsmile \mu_1)\smallsmile \mu_2)\ldots \smallsmile\mu_{n-1})\not=0.\] 
  \end{itemize}
\end{definition}

Note that the relative cup-length does not depend on the choice of the attracting local Morse datum by Proposition~\ref{prop:all}.

Now let $f$ be a function, $g$ a metric and $N$ an isolating neighborhood for the gradient flow $\varphi$ of $(f,g)$. We define the relative cup-length $Y(f,g,N)$ to be the relative cup-length $Y(\cQ^\alpha)$ where $\cQ^\alpha=(f^\alpha,g^\alpha,N,\ko^\alpha)$ is any local Morse datum whose gradient flow is isolated homotopic to the flow $\varphi$. Standard continuation arguments show that this does not depend on the choice of the local Morse datum. The proof of the following proposition closely follows~\cite{AH}.

\begin{theorem}
  \label{thm:main}
Let $f$ be a function, $g$ a metric and $N$ an isolating neighborhood for the gradient flow $\varphi$ of with respect to $g$. Then $f$ has at least $n=Y(f,g,N)$ critical points in $N$. 
\end{theorem}

\begin{proof}
  As the claim is trivial for $n=0$, we henceforth assume that $n>0$. Take a sequence of local Morse data $\cQ^{\alpha_j}=(f^{\alpha_j},g^{\alpha_j},N,\ko^{\alpha_j})$, $j\in\mN$, such that the gradient flow of each $\cQ^{\alpha_j}$ is isolated homotopic to the gradient flow $\varphi$, and such that $f^{\alpha_j}\rightarrow f$ and $g^{\alpha_j}\rightarrow g$ in the $C^\infty$ topology.
  Let us first consider the case $n=1$, i.e., we need to show that $f$ has a critical point. As $HM^*(\cQ^{\alpha_j})\not=0$ for each $\alpha_j$, there exists a critical point $x_j\in \crit f^{\alpha_j}$. As each $x_j\in N$ and $N$ is compact, there exists a convergent subsequence of $x_j$ that converges to $x\in N$. As $f^{\alpha_j}$ converges to $f$, it follows that $x$ is a critical point of $f$.
  
Now suppose that $n\geq 2$. % Let $\cQ^{\beta_i}$ for $1\leq i\leq n-1$ be attracting local Morse data, and assume that for each $i$ and $j$ the local Morse data $\cQ^{\alpha_j},\cQ^{\beta_i},\cQ^{\alpha_j}$ are mutually transverse. 
Let $\cQ^\beta$ be an attractive local Morse datum which is transverse to each $\cQ^{\alpha_j}$. Without loss of generality we may assume that $f$ has only finitely many critical points in $N$, as otherwise the claimed inequality obviously holds. Then we may assume that $\crit f^{\alpha}\cap \overline{W^s(y;\cQ^{\beta})}=0$ for all critical points $y$ of $f^{\beta}$ of index $\ind{y}\geq 1$. As the relative cup-length is $n\geq 2$, we have, for every $j$ and $1\leq i\leq n-1$, critical points $x_1^j,\ldots, x_n^j \in \crit f^{\alpha_j}$, critical points $y_i^j\in \crit f^{\beta}$ and points $z_i^j\in W(x^j_{i+1},y^j_i,x^j_i)$, such that
  \[
    f^{\alpha_j}(x_1^j)< f^{\alpha_j}(z_1^j)< f^{\alpha_j}(x_2^{j})< f^{\alpha_j}(z_2^j)<\ldots <f^{\alpha_j}(x_{n-1}^{j})< f^{\alpha_j}(z_{n-1}^j)< f^{\alpha_j}(x_n^{j}),
  \]
  see Figure~\ref{fig:cup} for the case $n=3$.  By passing to subsequences we may assume that each $x_i^j$ and $z^j_i$ converges in $N$ as $j\rightarrow \infty$ to points $x_i$ and $z_i$ respectively. In the limit we have a priori that
  \begin{equation}
    \label{eq:ineq}
    f(x_1)\leq f(z_1) \leq f(x_2) \leq f(z_2) \leq\ldots \leq f(x_{n-1}) \leq f(z_{n-1})\leq f(x_n).
\end{equation}
Note that all the $x_i$'s must be critical points of $f$, as the $x_i^j$ are critical points of $f^{\alpha_j}$ and $f^{\alpha_j}$ converges to $f$. We now argue that all the inequalities in Equation \eqref{eq:ineq} are strict, which means that the $x_i$'s are different critical points. By the assumption $\crit f^{\alpha}\cap \overline{W^s(x;\cQ^{\beta})}=0$, we see that the points $z_i$ are not critical points of $f$. But then for $t$ sufficiently small, $\varphi^{\alpha_j}_t(z^j_i)$ converges to $\varphi_t(z_i)$ and $\varphi^{\alpha_j}_{-t}(z^j_i)$ converges to $\varphi_{-t}(z_i)$. Thus
$f(x_i)\leq f(\varphi_t(z_i))<f(z_i)<f(\varphi_{-t}(z_i)\leq f(x_{i+1})$ for all $1\leq i\leq n-1$, which means that indeed all the inequalities in \eqref{eq:ineq} are strict. Consequently, $f$ has $n=Y(f,g,N)$ different critical points in $N$. 
\end{proof}

\begin{figure}
  \includegraphics[width=.3\textwidth]{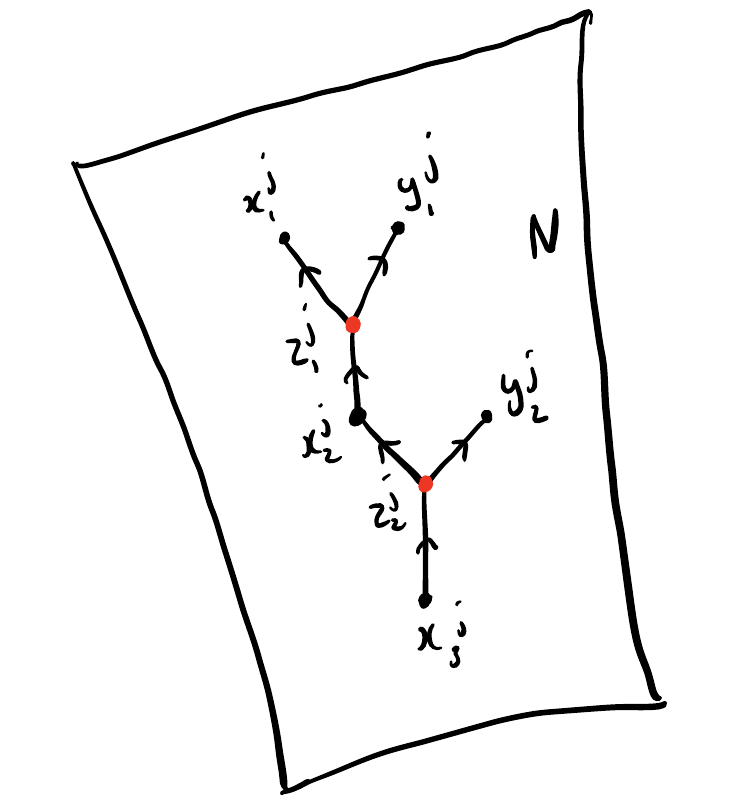}
  \caption{A sketch of the main idea of the proof of Theorem \ref{thm:main}. A non-zero relative cup length for every $j$ means that there exist non-vacuous diagrams in $N$ as in the figure (here for $n=3$). The points $x_1^j$ up to $x_n^j$ converge to (different) critical points $x_1,\ldots x_n$ of $f^\alpha$ as $j\rightarrow \infty$. }
  \label{fig:cup}
\end{figure}

\begin{remark}
It is possible to define different relative cup-lengths by changing the attracting local Morse data to other classes of Morse data. However these different cup-lengths do not yield the existence of more critical points, which can be seen as follows. Instead of taking $\cQ^{\beta}$ in Definition~\ref{def:relcup} to be an \emph{attracting} Morse data, let us consider any local Morse datum $\cQ^{\gamma}$, such that $\cQ^\alpha,\cQ^\gamma, \cQ^\alpha$ is isolation compatible. Denote by $Y$ the relative cup-length as in Definition~\ref{def:relcup}, and by $Y'$ the relative cup-length where $\cQ^{\beta}$ is changed to $\cQ^{\gamma}$ in Definition~\ref{def:relcup}. Then Theorem~\ref{thm:main} also holds for the relative cup-length $Y'$, i.e.~$Y'$ gives a lower bound on the number of critical points. However, $Y'\leq Y$, and thus this does not yield the existence of further critical points. To see the claimed inequality we observe that we have, under suitable genericity conditions, a commutative diagram
\begin{equation}
    \label{eq:ugly}
    \begin{split}
    \xymatrix{HM^{k}(\cQ^{\alpha})\times HM^{l_1}(\cQ^{\gamma})\times\ldots\times HM^{l_{Y'}}(\cQ^{\gamma})\ar[d]\ar[r]&HM^{k+l}(\cQ^{\alpha})\\
HM^{k}(\cQ^{\alpha})\times HM^{l_1}(\cQ^{\beta})\times\ldots\times HM^{l_{Y'}}(\cQ^\beta)\ar[ru]&
}
\end{split}
\end{equation}
In this diagram the horizontal and diagonal maps are the repeated cup products and the vertical map is induced by the identity. As $\cQ^\beta$ is attracting, the identity mapping is isolated~\cite{RV} (seen as a map $\mathrm{id}^{\gamma\beta}:\cQ^\beta\rightarrow \cQ^\gamma$), which means that the vertical map is well defined in cohomology.  To see that the diagram commutes it suffices to expand the cup-products as in Diagram~\eqref{eq:naturality}. The horizontal map in Diagram~\eqref{eq:ugly} is zero if the diagonal map is zero. This implies that $Y'\leq Y$. 
\end{remark}

    % \xymatrix{HM^{k}(\cQ^{\alpha})\times HM^{l_1}(\cQ^{\gamma_1})\times\ldots\times HM^{\gamma_{Y}}(\cQ^{\beta_Y})

\section{Example}

\label{sec:example}
In this section we provide an explicit example of the computation of the absolute and relative cup-length of Morse functions on $\mR\mP^n\times \mR$ and show that the relative cup length gives better lower bounds. We take as coefficients $\mZ/2\mZ$, although it would not be too much more effort to do these computations over $\mZ$. As the orientation $\ko$ is canonical over $\mZ/2\mZ$, we simplify our notation in this section by ignoring $\ko$. 

Let $R$ be a symmetric $(n+1)\times (n+1)$ matrix whose eigenvalues are all distinct. Denote these eigenvalues by $\lambda_0<\ldots <\lambda_n$ and by $p_0,\ldots ,p_n$ corresponding eigenvectors of norm one. Equip $S^n\subset \mR^{n+1}$ with the standard metric and define a function $f:S^n\rightarrow \mR$ by $f(x)=\frac{1}{2}\langle x,Rx\rangle$. Let us compute the gradient of $f$ and its flow. We have that
\[
  df(x)y=\frac{1}{2}(\langle y,R x\rangle+\langle x,R y\rangle)=\langle Rx,y\rangle
\]
as $R$ is symmetric. The gradient of $f$ on $\mR^{n+1}$ is therefore $\nabla^{\mR^{n+1}} f=Rx$, and the gradient of $f$ on $S^n$ is the projection of this gradient to the tangent bundle of $S^n$. Thus
\[\nabla f:=\nabla^{S^n} f(x)=Rx-\langle Rx,x\rangle x=(R-2f(x)) x.\]
We see that $x\in S^n$ is a critical point for $f$ if and only if $x$ is an eigenvector of norm one of $R$ with eigenvalue $2f(x)$, see Fig~\ref{fig:sphere}. We now show that $f$ is Morse. The Hessian of $f$ at a critical point $x$, seen as an operator $T_xS^n\rightarrow T_xS^n$, is given by $X\mapsto \nabla_X \nabla f$. But $\nabla_X \nabla f$ is the projection of the directional derivative (in $\mR^{n+1}$) of $\nabla f$ in the direction $X$ to the tangent space of $S^n$. We therefore see that
\[
  \nabla_X \nabla f=RX-(2df(x)X)x-2f(x)X=RX-2f(x)X
\]
Thus $\hess f(x)=R-2f(x)I$, where $I$ denotes the identity operator on $T_xS^n$. Let $k$ be such that $x=p_k$ or $x=-p_k$. The tangent space $T_xS^n$ is orthogonal to the eigenspace of $R$ spanned by $p_k$. Hence $\hess f(x)$ is invertible on $T_xS^n$ and the eigenvalues of $\hess f(x)$ are $\lambda_l-\lambda_k$ for $0\leq l\not=k\leq n$. The negative eigenvalues of $\hess f$ are therefore $\lambda_0-\lambda_k,\ldots, \lambda_{k-1}-\lambda_k$. Thus the critical point $x$ is non-degenerate and the Morse index is $k$. As this holds for all critical points, $f$ is Morse.

It is readily seen that the negative gradient flow of $\nabla f$ is given by
\[
\varphi_t(x)=\frac{\exp(-t R) x}{\norm{ \exp(-tR) x}}.
\]
This flow is the projection of the flow of $-Rx$ on $\mR^{n+1}\setminus\{0\}$. To study the long time behaviour of $\varphi$, let $x\in S^n$ and write $x=\sum_{i=0}^n a_i p_i$ in terms of the basis of eigenvectors. Then $\exp(-t R) x=\sum_{i=0}^n a_ie^{-t\lambda_i}p_i$. Let $k$ be the smallest number such that $a_i=0$ for all $i<k$. Then
\[
\lim_{t\rightarrow \infty}\varphi_t(x)=\lim_{t\rightarrow \infty}\frac{\sum_{i=k}^na_ie^{-(\lambda_i-\lambda_k)t}p_i}{\norm{\sum_{i=k}^na_ie^{-(\lambda_i-\lambda_k)t}p_i}}=\frac{a_k}{\vert a_k\vert}p_k=\sign(a_k) p_k.
\]
Similarly, if $l$ denotes the largest number such that $a_i=0$ for all $i>l$ then $\lim_{t\rightarrow -\infty}\varphi_t(x)=\sign (a_l) p_l$. Thus the stable and unstable manifolds are
\begin{align*}
    W^{u}(\pm p_k)&=\{a_0p_0+\ldots+a_k(\pm p_k)\,|\, a_0^2+\ldots +a_k^2=1, a_k>0\}\\
    W^{s}(\pm p_k)&=\{a_k(\pm p_k)+\ldots +a_n p_n\,|\, a_k^2+\ldots +a_n^2=1, a_k>0\},
\end{align*}
see Figure~\ref{fig:cup}. The stable and unstable manifolds are transverse, hence the system is of Morse-Smale type and the Morse cohomology is well-defined. For two critical points with consecutive indices we find that $W(\pm p_k,(\pm)'p_{k+1})$ has one connected component. This allows us to compute the Morse cohomology. Recall that we work over $\mZ/2\mZ$ so there are no choices of orientation involved.\\
The Morse cohomological complex of $f$ is generated by the two generators $\eta^{p_k}$ and $\eta^{-p_k}$ in each degree $0\leq k\leq n$, and the differential is $\delta \eta^{p_k}=\eta^{p_{k+1}}+\eta^{-p_{k-1}}$. The cohomology is non-trivial only in degree $n$ and degree $0$ and is isomorphic to $\mZ_2$ (generated by $\eta^{p_0}+\eta^{-p_0}$ and $\eta^{p_n}$ respectively). Thus there is no interesting cup product.\\  
The antipodal action $x\mapsto -x$ is a free discrete and proper action with quotient $\mR\mP^n$. The function $f$ and the metric on $S^n$ are invariant under this action and hence descend to $\mR\mP^n$. Henceforth we also denote by $f$ the function on the quotient. The function $f$ has a unique critical point $x_k=[p_k]$ for each index $0\leq k\leq n$, as  $[p_k]=[-p_k]$. The differential of $\delta$ is trivial as there are two connecting orbits between critical points of consecutive index, and we work over $\mZ/2\mZ$. Consequently, the Morse cohomology groups are isomorphic to $\mZ/2\mZ$ for each degree $0\leq i\leq n$.\\
Let us now compute the cup-product. For this we need to work with two different Morse functions. Let $R^\beta$ and $R^\alpha$ be two different symmetric matrices that are in general position, i.e., there is no linear dependence between less than $n$ eigenvectors of $R^\beta$ and $R^\alpha$ simultaneously. Let $f^\beta([x])=\frac 12\langle x,R^\beta x\rangle$ and $f^\alpha([x])=\frac 12\langle x,R^\alpha x\rangle$ two Morse functions on $\mR \mP^n$ as above. The Morse cohomologies are generated by the critical points $[p^\beta_i]$ and $[p^\alpha_i]$. 

We are interested in the cup product $\smallsmile:HM^k(\cQ^\alpha)\times HM^l(\cQ^\beta)\rightarrow HM^{k+l}(\cQ^\alpha)$. For this we need to compute the intersection (in $\mR\mP^n$)
\[
W^u([p^\alpha_{k+l}];\cQ^\alpha)\cap W^s([p^\alpha_k];\cQ^\alpha)\cap W^s([p^\beta_l];\cQ^\beta).
\]
This is the projectivization of the intersection of linear subspaces
\[\mathrm{span}\{p_k^\alpha,\ldots p_{k+l}^\alpha\}\quad\text{and}\quad \mathrm{span}\{p_l^\beta,\ldots p_{n}^\beta\}\] of $\mR^{n+1}$. But under the genericity assumption on the eigenspaces of $R^\alpha$ and $R^\beta$,  the intersection is a one dimensional linear subspace and hence its projectivization consists of a single point, cf.~Figure~\ref{fig:sphere}. We find that $\eta^{[p^\alpha_k]}\smallsmile \eta^{[p^\beta_l]}=\eta^{[p^\alpha_{k+l}]}$, which determines all the cup products and shows that the (absolute) cup length is $n+1$.

Now let $F^\alpha,F^\beta,F^\gamma:\mR\mP^n\times \mR\rightarrow \mR$ be given by $F^\alpha(x,y)=f^\alpha(x)-y^2$, $F^\beta(x,y)=f^\beta(x)+(y-\frac{1}{2})^2$, and $F^\gamma(x,y)=f^\beta(x)-(y-\frac{1}{2})^2$. Each of these functions has exactly one critical point of index $k$ and these are $x^\alpha_{k}=([p^\alpha_{k-1}],0)$, $x^\beta_{k}=([p^\beta_k],\frac{1}{2})$ and $x^\gamma_k=([p^\alpha_{k-1},\frac{1}{2})$. Note the shift in the index. Let $N=\mR\mP^n\times[-1,1]$.  Fix the product metric $g$ on $\mR\mP^n\times \mR$. Let $\cQ^\alpha=(N, F^\alpha,g,\ko^\alpha), \cQ^\beta=(N, F^\beta,g,\ko^\beta)$ and $\cQ^\gamma=(N, F^\gamma,g,\ko^\gamma)$. Then the Morse cohomology $HM^*(\cQ^\beta)$ is $\mZ_2$ generated by $\eta^{x^\beta_k}$ for each $0\leq k\leq n$, and the Morse cohomologies $HM^*(\cQ^\alpha)$ and $HM^*(\cQ^\gamma)$ are generated by $\eta^{x^\alpha_k}$ and $\eta^{x^\gamma_k}$ for each $1\leq k\leq n+1$.

The relative cup length can be computed as follows. The intersection
\[
W^u(x^\alpha_{k+l};\cQ^\alpha)\cap W^s(x^\alpha_{k};\cQ^\alpha)\cap W^s(x^\beta_{l};\cQ^\beta)
\]
is equal to the projectivization in the first coordinate of
\[\mathrm{span}\{p_k^\alpha,\ldots p_{k+l}^\alpha\}\times\{0\},\quad\text{and}\quad \mathrm{span}\{p_l^\beta,\ldots p_{n}^\beta\}\times\{0\}
\]
which consists of one point. Thus we find that $\eta^{x^\alpha_k}\smallsmile \eta^{x^\beta_l}=\eta^{x^\alpha_{k+l}}$, and the relative cup-length is $Y(\cQ^\alpha)=n+1$.\\
The absolute cup length vanishes however. Indeed, the intersection 
\[
W^u(x^\alpha_{k+l};\cQ^\alpha)\cap W^s(x^\alpha_{k};\cQ^\alpha)\cap W^s(x^\gamma_{l};\cQ^\gamma)
\]
is empty as $W^u(x^\alpha_{k+l};\cQ^\alpha)\cap W^s(x^\alpha_{k};\cQ^\alpha)$ is contained in $\mR\mP^n\times\{0\}$ while $W^s(x^\gamma_{l};\cQ^\gamma)$ is contained in $\mR\mP^n\times\{\frac{1}{2}\}$. Thus the cup product $\smallsmile:HM^k(\cQ^\alpha)\times HM^l(\cQ^\gamma)\rightarrow HM^{k+l}(\cQ^\alpha)$ is zero and consequently the (absolute) cup-length is zero as well. 

\begin{figure}
  \includegraphics[width=.5\textwidth]{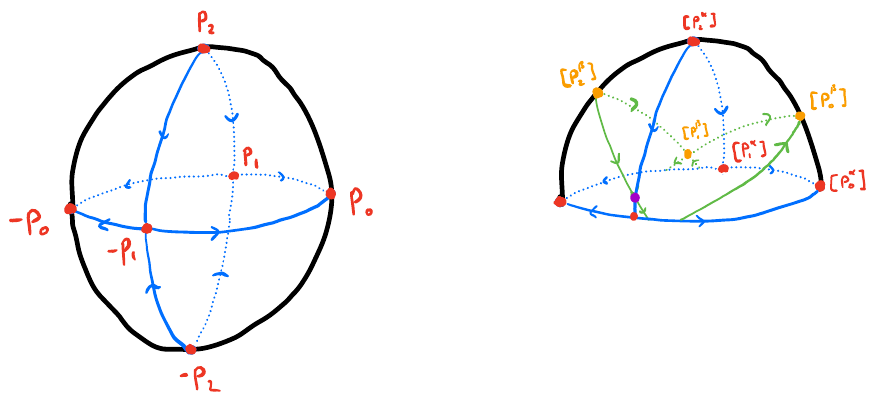}
  \caption{On the left is a sketch of the gradient flow of $f(x)=\frac{1}{2}\langle x, R x \rangle$ on the two sphere depicted. On the right the gradient flow of the projectivization of two such Morse functions on $\mR\mP^2$ is shown. The intersection \[W^u([p_2^\alpha] ;\cQ^\alpha)\cap W^s([p_1^\alpha];\cQ^\alpha)\cap W^s([p_1^\beta];\cQ^\beta)\] consists of a single point drawn in purple. This shows that $\eta^{[p_1^\alpha]}\smallsmile \eta^{[p_1^\beta]}=\eta^{[p_2^\alpha]}.$}
  \label{fig:sphere}
\end{figure}

% \todo{Do we want to say something about associativity?}
% \section{Another appendix: breaking}

% 2\begin{proposition}
%   Assume gradient flow is Morse smale on compact manifold. Let $p_k\in W^u(x)$ a sequence. Then there exists a critical point $x'$ and a convergent subsequence $p_k\rightarrow p$ with $p\in W^u(x')$. If $x\not=x'$ then $\ind {x'}<\ind x$. Similarly if $p_k\in W^s(y)$, then there exits a critical point $y'$ and a convergent subsequence $p_k\rightarrow p$ with $p'\in W^s(y')$. If $y\not=y'$, then $\ind {y'}>\ind y$. 
% \end{proposition}
\section{Appendix: Orientations and multiple intersections.}
\label{sec:appendix}

The aim of this appendix is to fix the orientation conventions used in this paper. Let us first consider a short exact sequence of finite dimensional vector spaces
\[
  0\rightarrow U\xrightarrow{i} V\xrightarrow{j} W\rightarrow 0. 
\]
If two of the three vector spaces $U,V$ and $W$ are oriented, this short exact sequence automatically orients the third. Let us see this explicitly in the case that $V$ and $W$ are oriented. Let $w_1,\ldots w_l$ be a basis of $W$ inducing the orientation of $W$. The orientation on $U$ is then defined to be as follows. Let $v_1,\ldots ,v_l\in V$ such that $j(v_i)=w_i$. A basis $u_1,\ldots ,u_k$ of $U$ is compatible with the induced orientation if $i(u_1),\ldots ,i(u_k),v_1,\ldots ,v_l$ is a basis compatible with the orientation of $V$.

Now if $X,Y$ are submanifolds in general position of a manifold $M$ and $X$ is oriented and $Y$ is cooriented, i.e. the normal bundle is oriented, then the intersection $X\cap Y$ is automatically oriented by the short exact sequence $0\rightarrow T(X\cap Y)\rightarrow TX\rightarrow TX/TY\rightarrow 0$ of vector bundles. Note that this orientation convention does not require that $M$ is oriented. Intersections of orientable submanifolds need not be orientable. Indeed, two copies of $\mR\mP^3\subset \mR\mP^4$ in general position intersect in an $\mR \mP^2$ which is not orientable. However $\mR\mP^3$ is orientable.

The story for multiple intersections is analogous. Submanifolds $X_1,\ldots, X_n\subset M$ are \emph{mutually transverse}, or in \emph{general position}, if for each $x\in \bigcap_{i=1}^n X_i$ the diagonal map
$\Delta:T_xM\rightarrow \oplus_{i=1}^nT_xM/T_xX_i$ that sends $y\in T_xM$ to $([y]_{T_xM/T_xX_1},\ldots, [y]_{T_xM/T_xX_n})$ is surjective. This symmetric condition implies that any $X_i$ is transverse to the intersection of any combination of the other $X_j$'s, which can be easily verified. The intersection of an oriented $X$ and a number of cooriented manifolds $Y_1,\ldots,Y_k$ in general position then automatically get oriented by viewing
$
  X\cap Y_1\cap \ldots \cap Y_k=\left(\left(\ldots\left(\left(X\cap Y_1\right)\cap Y_2\right)\ldots \right)\cap Y_k\right)
$ and using the convention above. With this orientation convention we have $X\cap Y_1\cap Y_2=(-1)^{\codim Y_1 \codim Y_2}X\cap Y_2\cap Y_1$.

Finally, if $\mR$ acts properly and freely on an oriented manifold $M$ then $M/\mR$ is oriented by the short exact sequence $0\rightarrow T\mR\rightarrow TM\rightarrow T M/\mR\rightarrow 0$, where $T\mR$ is canonically oriented in the positive direction.

% \tableofcontents
%\todo{Bibliography is a mess. Use bibtex?}
\newpage    
\vspace*{1.3cm}

\begin{minipage}{1.2\textwidth}
\begin{minipage}{0.4\textwidth}
Thomas O. Rot\\
Department of Mathematics\\
Vrije Universiteit Amsterdam\\
De Boelelaan 1111\\
1081 HV Amsterdam,\\
The Netherlands\\
t.o.rot@vu.nl\\\\
Nils Waterstraat\\
Martin-Luther-Universit\"at Halle-Wittenberg\\
Naturwissenschaftliche Fakult\"at II\\
Institut f\"ur Mathematik\\
06099 Halle (Saale)\\
Germany\\
nils.waterstraat@mathematik.uni-halle.de

\end{minipage}
\hfill
\begin{minipage}{0.6\textwidth}
Maciej Starostka\\
Martin-Luther-Universit\"at Halle-Wittenberg\\
Naturwissenschaftliche Fakult\"at II\\
Institut f\"ur Mathematik\\
06099 Halle (Saale)\\
Germany\\
and\\
Institute of Applied Mathematics\\
Faculty of Applied Physics and Mathematics\\
Gda\'{n}sk University of Technology\\
Narutowicza 11/12, 80-233 Gda\'{n}sk, Poland\\
maciejstarostka@gmail.com

\end{minipage}
\end{minipage}


\thebibliography{999999}

\bibitem[AS]{AS} A. Abbondandolo, M. Schwarz, \textbf{Floer homology of cotangent bundles and the loop product}, Geometry \& Topology (2010), 14(3), 1569--1722

\bibitem[AH]{AH} P. Albers, D. Hein, \textbf{Cuplength estimates in Morse cohomology}, J. Topol. Anal. (2016), 8(2) 243--272

\bibitem[DzGU]{KGUss} Z. Dzedzej, K. G\c{e}ba, W. Uss, \textbf{The Conley index, cup-length and bifurcation}, Journal of Fixed Point Theory and Applications (2011), 10(2), 233--252

\bibitem[Hi]{Hi} M.W. Hirsch, \textbf{On imbedding differentiable manifolds in euclidean space}, Ann. of Math.  (1961), 73, 566--571

\bibitem[IRSV]{IRSV} M. Izydorek, T.O. Rot, M. Starostka, M. Styborski, R.C.  Vandervorst, \textbf{Homotopy invariance of the Conley index and local Morse homology in Hilbert spaces}, Journal of Differential Equations (2017), 263(11), 7162--7186

\bibitem[KM]{KM} P. Kronheimer, T. Mrowka, \textbf{Monopoles and three-manifolds,} Cambridge University Press, 2007, ISBN: 978-0-521-88022-0

  
\bibitem[R]{R} T.O. Rot, \textbf{Morse-Conley-Floer Homology}, 2014, ISBN 978-94-6259-399-2

\bibitem[RV]{RV} T.O. Rot, R.C. Vandervorst, \textbf{Morse-Conley-Floer homology}, Journal of Topology and Analysis (2014), 6(3), 305--338

  \bibitem[RV2]{RV2} T.O. Rot, R.C. Vandervorst, \textbf{Functoriality and duality in Morse-Conley-Floer homology}, J. Fixed Point Theory Appl. (2014), 16(1-2), 437–-476 

\bibitem[StWa]{StWa} M. Starostka, N. Waterstraat, \textbf{The E-Cohomological Conley Index, Cup-Lengths and the Arnold Conjecture on $T^{2n}$}, Advanced Nonlinear Studies (2019), 19(3), 519--528
\end{document}